\documentclass[12]{amsart}

\usepackage{amsthm}
\usepackage{amsmath} \usepackage{amsfonts}
\usepackage{amssymb} \usepackage{latexsym} \usepackage{enumerate}
\usepackage{mathtools}
\mathtoolsset{showonlyrefs}
\usepackage{todonotes}
\usepackage{bbm}
\usepackage{hyperref}

\usepackage{mathrsfs}
\usepackage[numbers]{natbib}

\usepackage[a4paper,left=3.3cm,right=3.3cm]{geometry}

\newtheorem{thm}{Theorem}[section]

\newtheorem{corollary}[thm]{Corollary}

\newtheorem{remark}[thm]{Remark}

\newtheorem{lemma}[thm]{Lemma}
\DeclareMathOperator*{\argmax}{arg\,max}

\DeclareMathOperator*{\E}{\mathbb{E}}

\begin{document}
\title{Merton's optimal investment problem with jump signals}
\thanks{P.~Bank is supported in part by the
GIF Grant 1489-304.6/2019 and by Deutsche Forschungsgemeinschaft through IRTG 2544. \\
L.~Körber is supported by Deutsche Forschungsgemeinschaft through IRTG 2544. }

   \author{Peter Bank \address{
 Department of Mathematics, TU Berlin. \\
 e.mail: bank@math.tu-berlin.de }}
    \author{Laura Körber \address{
 Department of Mathematics, TU Berlin. \\
 e.mail: koerber@math.tu-berlin.de }}
\date{\today}

 \maketitle \markboth{Peter Bank and Laura Körber}{Optimal investment with jump signals}
\begin{abstract}
This paper presents a new framework for Merton's optimal investment problem which uses the theory of Meyer $\sigma$-fields to allow for signals that possibly warn the investor about impending jumps. With strategies no longer predictable, some care has to be taken to properly define wealth dynamics through stochastic integration. By means of dynamic programming, we solve the problem explicitly for power utilities. In a case study with Gaussian jumps, we find, for instance, that an investor may prefer to disinvest completely even after a mildly positive signal. Our setting also allows us to investigate whether, given the chance, it is better to improve signal quality or quantity and how much extra value can be generated from either choice.
\end{abstract}
\begin{description}
\item[Mathematical Subject Classification (2020)] 91G10, 91B16, 60H05 
\item[Keywords] Optimal investment, jump signal, Meyer $\sigma$-field, stochastic integration for non-predictable integrands
\end{description}

\section{Introduction}

In continuous-time stochastic optimal control problems for systems with exogenous jumps, the standard approach is to assume that the controller can only react to these afterwards, making the class of predictable policies a natural choice; see, e.g., \cite{oksendalsulem}. In many applications, though, one can readily imagine that such jumps are not coming as a complete surprise, but that beforehand the controller receives some information about their imminence and size, affording her the opportunity to pro-actively intervene. 

The present paper illustrates the possibilities and scope of such signal-driven policies in the classical optimal investment problem with jumps, going back to \cite{Merton71}. Specifically, we will assume that the stock price is driven by a L{\'e}vy process, but, contrary to the standard, purely predictable case, we will allow that, whenever there is a jump, say at time $t$, the investor may receive a signal $\Delta Z_t\not=0$, for instance, on its size. Given this signal, the investor can then adjust her holdings to be best prepared for the price jump $\Delta S_t$. 

A mathematically rigorous description of the resulting class of policies is made possible by the Meyer $\sigma$-field $\Lambda = \mathcal{P}\vee \sigma(Z)$ obtained as the augmentation of the usual predictable $\sigma$-field $\mathcal{P}$ by the $\sigma$-field generated by the signal process $Z$. We will clarify the formal structure of admissible policies and rigorously establish wealth-dynamics by constructing stochastic integrals beyond the predictable integrands classically considered. We will also show how these dynamics allow one to derive the Hamilton-Jacobi-Bellman (HJB) equation for the corresponding stochastic control problem. As one would expect, it takes the form of an integro-partial differential equation. A novel feature in our HJB-equation is the occurrence of an iterated integral that straddles a sup-operator in reflection of the controller's ability to react to a dynamically generated signal. 

As in Merton's classical paper, an explicit solution of the HJB-equation is possible for power utility functions. Here, optimal investment policies turn out to hold a constant fraction of wealth to invest in the stock when there is no signal; upon reception of a signal, say $\Delta Z_t=z\not=0$, though, the investor adjusts her position to take advantage of the impending jump in a way that crucially depends on what this signal entails for the stock price. The respectively optimal investment fractions are described by simple one-dimensional concave optimization problems. 

In a case study, we consider an investor who, with some probability $p$, will receive a signal of the imminent price jump specified through a correlated pair of jointly normally distributed random variables.  Here, the optimal investment fractions with and without a signal can be computed numerically. It turns out, for instance, that the investor will almost completely disinvest even given a mildly positive signal. This can be viewed as a hedge against an impending jump motivated by the investor's risk aversion. We also compute the maximal expected utility and discuss how it depends on both the probability $p$ of receiving a signal when there is a jump and on the correlation $\rho$ between such a signal and the jump shock. For a typical degree of risk aversion, we find that improvements in signal quality (higher correlation $\rho$) have a similar impact as improvements in signal quantity (higher signal probability $p$); for high risk aversion, however, this trade-off becomes skewed.

Since Merton's celebrated seminal work \cite{Merton71}, continuous-time optimal investment problems have been widely studied in the literature. Rather than embarking on the impossible task of reviewing all what has been done there, let us highlight a few directions which we consider closest to our efforts here. One way to interpret our jump signal setting is that our investor has access to privileged ``inside'' information, but can trade freely without impacting the underlying stock price. Seminal papers in this direction of ``small insiders'' include \cite{PikovskyKaratzas:96} and \cite{AmendingerImkellerSchweizer:98}, but there the inference problem is due to an initial enlargement of the investor's information flow, not a dynamic one as here; as a result, their inference problem is much more delicate than ours. Efforts to endogenize the impact on stock prices by the insider's trades have spurred a tremendous amount of work starting with \cite{Kyle:85, Back:92}; see \cite{EkrenEtAl:21} for a recent connection of this equilibrium problem to optimal transport and for further references. 

The use of Meyer $\sigma$-fields in stochastic optimization has been pioneered in the general theory of optimal stopping due to~\cite{ElKaroui:81}. Meyer $\sigma$-fields were introduced by~\cite{Lenglart} who used them to unify Meyer's treatment of his predictable and optional section and projection theorems. In a singular stochastic control setting, \cite{bankbesslichmodelling} use Meyer $\sigma$-fields in a general study of irreversible investment problems and in an explicitly worked out Poissonian case study. Their approach is through convex analytic methods though, making the present paper, to the best of our knowledge, the first to use dynamic programming methods in this context.

The paper is structured as follows: Section~\ref{sec:wealthprocesses} introduces the asset price dynamics, clarifies the structure of our Meyer-strategies and then proceeds to specify the induced wealth-dynamics by developing a suitable stochastic integral. Section~\ref{sec:utilitymaximization} then uses these dynamics to formulate an optimal investment problem with jump signals, to derive the corresponding Hamilton-Jacobi-Bellman equation and then discusses its solution in the case of power utilities. It concludes the paper with a numerical case study for Gaussian jumps and signals.

\section{Wealth processes with jump signals}~\label{sec:wealthprocesses}

In this section we will describe the financial market in which our investor operates as well as the information flow on which she can base her investment choices. Specifically, we will reach beyond the classical predictable case and allow her to draw on a signal process which reveals information on impending jumps. This will lead to signal dependent position limits ensuring nonnegative wealth and, more fundamentally, to the definition of stochastic integrals for certain non-predictable integrands in order to properly describe wealth dynamics in the first place.

\subsection{The financial market and wealth dynamics for predictable strategies}
We consider an investor with initial capital $x_0>0$ who, over some finite period $[0,T]$ wants to invest in a market with two assets, a riskless bond $S^0$ bearing interest at the constant rate $r \in \mathbb{R}$ and a stock $S^1$. The stock's returns are driven both by a standard Brownian motion~$W$ and a homogeneous Poisson measure~$N$ (cf.\ Def.\ II.1.20 in \cite{jacodshiryaev}) on a Polish mark space $(E,\mathcal{E})$ with $\sigma$-finite intensity measure $dt \otimes \nu(de)$. The drivers $W$ and $N$ are independently defined on the joint probability $(\Omega, \mathcal{A},\mathbb{P})$ and yield the model's (right-continuous) filtration $(\mathcal{F}_t)$ via the completions of $\sigma((W_s, \;N([0,s] \times A))_{s \in [0,t],\ A\in \mathcal{E}})$, $t \in [0,T]$. The asset price dynamics in this paper are specified by
\begin{equation}\label{assetdynamics}
\begin{aligned}
S^0_0=s^0>0, \quad dS^0_t&=rS^0_{t-}dt, \quad t\in[0,T]\\
S^1_0=s^1>0, \quad dS^1_t&=S^1_{t-}\left[\kappa dt+\sigma dW_t+\int_E \eta(e)\tilde{N}(dt,de)\right],\quad t\in[0,T].
\end{aligned}
\end{equation}
Here, $\kappa\in\mathbb{R}$ is considered as a drift in stock prices and $\sigma>0$ is their constant volatility; the measurable mapping $\eta:E \to \mathbb{R}$ specifies how a mark $e\in E$ set by the Poisson measure $N$ translates into a return jump $\eta(e)$ in the stock. The random measure $\tilde{N}(dt,de) := N(dt,de)-1_{\{|\eta|\leq 1\}}(e)dt \otimes \nu(de)$ compensates the small jumps in just the right way to allow us to ensure by standard arguments the existence of a unique strong solution to~\eqref{assetdynamics} under the standing  assumption that
\begin{align}\label{etasquareintegrable}
    \int_E \left(\eta(e)^2 \wedge 1\right) \nu(de)<+\infty.
\end{align}
Assuming in addition that
\begin{align}\label{etalowerbound}
    \nu(\{\eta < -1\})=0
\end{align}
will ensure that the process $S^1$ remains nonnegative, as customary for stock prices.


It will be convenient to describe the investor's investment strategies by the dynamically adjustable proportion $\tilde{\phi}=(\tilde{\phi}_t)$ of her wealth she is investing in the risky asset. The investor's resulting wealth $X^{\tilde{\phi}}$ will then evolve according to
\begin{equation}\label{wealthdynamics1}
\begin{aligned}
X^{\tilde{\phi}}_0&=x_0, \\ dX^{\tilde{\phi}}_t&=(1-\tilde{\phi}_t)X^{\tilde{\phi}}_{t-}\frac{dS^0_t}{S^0_{t-}}+\tilde{\phi}_tX^\phi_{t-}\frac{dS^1_t}{S^1_{t-}}=X^{\tilde{\phi}}_{t-}dR^{\tilde{\phi}}_t
, \quad t \in [0,T].
\end{aligned}
\end{equation}
for the cumulative return process
\begin{align}
    R_t^{\tilde{\phi}} = \int_0^t(r+\tilde{\phi}_s(\kappa-r))ds+\int_0^t\sigma \tilde{\phi}_s dW_s+\int_{[0,t]\times E} \tilde{\phi}_s \eta(e)\tilde{N}(ds,de).
\end{align}
It is well-known that these dynamics will have a unique solution for $\tilde{\phi}$ which are measurable with respect to the predictable $\sigma$-field $\mathcal{P}$ generated by the continuous $(\mathcal{F}_t)$-adapted processes and which satisfy the integrability condition
\begin{align}
    \int_0^T \tilde{\phi}^2_s ds<+\infty.
\end{align}
In order to understand when wealth processes remain nonnegative throughout for such strategies, let us note first that without jumps, i.e.\ if $\nu(\{\eta\not=0\})=0$, the investor can take any position in
\begin{align}
    \Phi:=\mathbb{R}.
\end{align}
With jumps, i.e.\ in case $\nu(\{\eta\not=0\})>0$, we observe that the jump bounds
\begin{align}
    \underline{\eta}:=\nu\text{-ess inf } \eta \geq -1  \text{ and } \overline{\eta}:=\nu\text{-ess sup } \eta 
\end{align}
allow us to describe the set of positions ensuring nonnegative wealth by
\begin{equation}\label{AdmSet_Predicable}
\Phi:=\begin{cases}\big[-1/\overline{\eta}\;,\;-1/\underline{\eta}\big],&\quad \text{if }\underline{\eta}<0<\overline{\eta}, \\
\big[-1/\overline{\eta}\;,\;+\infty\big),&\quad \text{if }0\leq\underline{\eta}\leq\overline{\eta},\\
\big(-\infty\;,\;-1/\underline{\eta}\big],&\quad \text{if }\underline{\eta}\leq\overline{\eta}\leq 0.
\end{cases} 
\end{equation}
The natural class of admissible predictable strategies is therefore
\begin{align}
    \mathcal{A}:=\left\{ \tilde{\phi} \text{ predictable with } \int_0^T \tilde{\phi}_t^2dt<+\infty \text{ and } \tilde{\phi}_t \in \Phi, \;t\in [0,T], \text{ a.s.}\right\}.
\end{align}

\subsection{Strategies with jump signals}

It is the aim of this paper to investigate how extra information on the stock price jumps will affect investments. Specifically, we consider an investor who can follow strategies that can also use a signal on impending jumps as given by the process
\begin{align}
    Z_t = \int_{[0,t] \times E} \zeta(e) \tilde{N}(ds,de), \quad t \in [0,T],
\end{align}
where $\zeta:E \to \mathbb{R}$ is measurable with $\int_E (\zeta(e)^2 \wedge 1) \nu(de)<\infty$. So, rather than merely being predictable, the investor's strategies will be allowed to be measurable with respect to the larger $\sigma$-field 
\begin{align}
    \Lambda := \mathcal{P} \vee \sigma(Z).
\end{align}
\begin{remark}
This is an example of what is known as a Meyer $\sigma$-field as introduced by \cite{Lenglart}; see \cite{bankbesslichlenglarts} for a brief survey. Depending on how much the signal value $\zeta(e)$ reveals about its pre-image $e$ from $E$, the field $\Lambda$ can be used to interpolate rather flexibly between the predictable $\sigma$-field $\mathcal{P}$ (no extra information beyond the standard framework, $\zeta \equiv 0$) and its optional counterpart $\mathcal{O}$ generated by the adapted càdlàg processes (conveying full information about jumps as they happen, $\zeta\equiv \eta$). 
\end{remark}

The strategic possibilities beyond predictable investment choices are made more tangible by the following lemma.

\begin{lemma}\label{LemmaRepresentation} A real-valued stochastic process $\tilde{\phi}$ is measurable with respect to the Meyer $\sigma$-field $\Lambda=\mathcal{P}(\mathcal{F})\vee \sigma(Z)$ if and only if it can be written in the form 
\begin{align}
\tilde{\phi}_t(\omega)=\phi_t(\omega,\Delta_t Z(\omega)), \quad (t,\omega) \in [0,T] \times \Omega,
\end{align} 
for some $\mathcal{P}(\mathcal{F})\otimes \mathcal{B}(\mathbb{R})$-measurable field $\phi:[0,T]\times \Omega\times \mathbb{R}\rightarrow \mathbb{R}$.
\end{lemma}
\begin{proof}
 This follows by a straightforward monotone class argument.
 \end{proof}

Our goal for the rest of this section will be to give proper meaning to the dynamics~\eqref{wealthdynamics1} for as many $\Lambda$-measurable $\tilde{\phi}=\phi(\Delta Z)$ as possible and to understand when they lead to nonnegative wealth processes $X^\phi:=X^{\tilde{\phi}} \geq 0$. 

\subsection{Positions maintaining nonnegative wealth}

Nonnegativity in our wealth dynamics will be obtained by conditions on the field $\phi=\phi_t(z)$ which ensure that no conceivable return shock $\eta(e)$ will correspond to a signal $z=\zeta(e)$ such that $\phi_t(z)\eta(e) < -1$. To formalize what return shocks are conceivable given a signal, let us disintegrate the jump measure $\nu$ in the form
\begin{align}\label{eq:disintegration}
  \nu(de \cap \{\zeta \not=0\}) = \int_{\zeta(E)\setminus\{0\}} K(z,de) \mu(dz) \text{ with } \mu := \nu \circ \zeta^{-1}.
\end{align}
Then the measure $\mu(dz)$ describes the frequency for receiving the signal $z$ and the kernel $K(z,de)$ describes the a posteriori probability of the jump mark $e$ conditional on having received the signal $z$; in particular, $K(z,E)=K(z,\{\zeta=z\})=1$. Of course, the case of no signal corresponds to $z=0$ and will need to be handled as well.

\begin{remark}
In our explicit case study below, signal and mark will be jointly normal and so the kernel $K(z,de)$ will describe the usual Bayesian update on the distribution of one marginal (the mark) given the other (the signal). In general, the possibility for such a disintegration is ensured because our mark space $E$ is Polish; cf. the discussion in \cite{jacodshiryaev}, p.~65. 
\end{remark}

Given an actual (i.e.\ a nonzero) signal, tight jump bounds are given by
\begin{align}\label{defetaboundssignal}
    \underline{\eta}(z):=K(z,.)\text{-ess inf } \eta \text{ and } \overline{\eta}(z):=K(z,.)\text{-ess sup } \eta, \quad z\in \zeta(E)\setminus\{0\}.
\end{align}
To rule out arbitrage opportunities (and thus degeneracy of our model), we impose the natural assumption
\begin{align}\label{noarbitragegivensignal}
    \underline{\eta}(z) < 0 < \overline{\eta}(z) \text{ for } \mu\text{-a.e. } z \in \zeta(E)\setminus\{0\}.
\end{align}
Note that if this condition is violated, for instance because $\mu(\underline{\eta} \geq 0) > 0$, the investor will receive with positive probability a signal which allows her to infer that stock prices can only jump upwards next; in such moments she can go long in the stock and thus pocket a riskless profit.  

Thus, under our henceforth standing assumption~\eqref{noarbitragegivensignal}, the interval
\begin{equation}\label{AdmSet_Signal}
\Phi(z):=\Big[-1/\overline{\eta}(z)\;,\;-1/\underline{\eta}(z)\Big], \quad z \in \zeta(E)\setminus\{0\},
\end{equation}
is identified as the range of positions $\phi_t(z)$ that, for a given signal $z \in \zeta(E)\setminus\{0\}$, will ensure $\eta(e) \phi_t(z) \geq -1$ (and thus nonnegative wealth after the impending jump) for $K(z,.)$-almost every mark $e \in E$.

Without a signal, the investor still learns something, namely that, if any, shocks can only emerge from marks placed in the set $\{\zeta=0\}$. Now, if $\nu(\{\zeta=0, \; \eta\not=0\})=0$, the investor can rule out unsignaled shocks almost surely and so can maintain nonnegative wealth by taking any position in
\begin{align}
    \Phi(0):=\mathbb{R}.
\end{align}
In the complementary case $\nu(\{\zeta=0,\eta\neq 0\})>0$, we introduce 
\begin{align}\label{etaboundsnosignal}
    \underline{\eta}(0)&:=\nu( . \cap \{\zeta=0,\eta\neq 0\}) \text{-ess inf } \eta \geq -1, \\ 
    \overline{\eta}(0)&:=\nu( . \cap \{\zeta=0,\eta\neq 0\})\text{-ess sup } \eta ,
\end{align}
which allow us to describe the set of positions ensuring nonnegative wealth by
\begin{equation}\label{AdmSet_Signal2}
\Phi(0):=\begin{cases}\big[-1/\overline{\eta}(0)\;,\;-1/\underline{\eta}(0)\big],&\quad \text{if }\underline{\eta}(0)<0<\overline{\eta}(0), \\
\big[-1/\overline{\eta}(0)\;,\;+\infty\big),&\quad \text{if }0\leq\underline{\eta}(0)\leq\overline{\eta}(0),\\
\big(-\infty\;,\;-1/\underline{\eta}(0)\big],&\quad \text{if }\underline{\eta}(0)\leq\overline{\eta}(0)\leq 0.
\end{cases} 
\end{equation}

\subsection{Cumulative returns with jump signals}\label{chapter_wealthprocess}

An extension from predictable to Meyer-measurable $\tilde{\phi}=\phi(\Delta Z)$ is nontrivial for dynamics such as~\eqref{wealthdynamics1} since the cumulative returns process
\begin{align}\label{returnsprocess}
    R^\phi_t :=\int_{[0,t]}\left(r+\tilde{\phi}_s(\kappa-r)\right)ds+\int_{[0,t]}\tilde{\phi}_s\sigma dW_s+\int_{[0,t] \times E}\tilde{\phi}_s\eta(e)\tilde{N}(ds,de)
\end{align}
involves stochastic integration. 

Actually, the only issues arise with parts of the last term. Indeed, for the $ds$-terms the distinction between $\tilde{\phi}$ and its counterpart without signal $\phi(0)$ is harmless because, differing only at countably many times, they belong to the same equivalence class of $ds$-integrands; in fact, the same argument works for the $dW_s$-terms due to the continuity of Brownian paths. Also the contributions from sizable return shocks, say of absolute size larger than~1, do not pose a problem for making sense of~\eqref{wealthdynamics1}: marks in $\{|\eta|>1\}$ occur only finitely many times almost surely because of~\eqref{etasquareintegrable} and the jump integral over $[0,t]\times\{|\eta|>1\}$ is just a finite sum. Moreover, unsignaled contributions from small jumps as collected in 
\begin{align}
    \int_{[0,t] \times \{\zeta=0,|\eta|\leq 1\}}\tilde{\phi}_s\eta(e)\tilde{N}(ds,de) = \int_{[0,t] \times \{\zeta=0,|\eta|\leq 1\}}\phi_s(0)\eta(e)\tilde{N}(ds,de)
\end{align}
are well-defined by classical stochastic integration theory provided that $\int_0^T\phi^2_s(0)ds<\infty$ almost surely (cf. \citep[p. 71ff.]{jacodshiryaev}).

The main issue is to give proper meaning to the contributions from small signaled jumps, i.e., to
\begin{align}
    I_t(\tilde{\phi})&=\int_{[0,t] \times \{\zeta\neq0,|\eta|\leq 1\}}  \tilde{\phi}_s\eta(e) \tilde{N}(ds,de)\\&=\int_0^t \tilde{\phi}_s d \int_{[0,s] \times \{\zeta\neq0,|\eta| \leq 1\}} \eta(e) \tilde{N}(dr,de), \quad t \in [0,T].
\end{align}
Indeed, $\int_{[0,.] \times \{\zeta\neq0,|\eta| \leq 1\}} \eta(e) \tilde{N}(dr,de)$ may be a martingale of unbounded variation (namely exactly when $\int_{\{\zeta\not=0, |\eta|\leq 1\}} |\eta(e)| \nu(de)=\infty$), but its jumps may be partially anticipated by the integrand $\tilde{\phi}=\phi(\Delta Z)$ due to the signal process $Z$, and so blow ups might occur even for bounded integrands.

In the simple case where $\tilde{\phi}=\phi(\Delta Z)$ almost surely differs from the predictable $\phi(0)$ only at finitely many (random) times, defining the integral $I(\tilde{\phi})$ is straightforward:
\begin{align}\label{defsimpleint}
    I_t(\tilde{\phi})&:=\int_{[0,t]\times \{\zeta\neq0,|\eta|\leq 1\}} \phi_s(0) \eta(e) \tilde{N}(ds,de)\\&\quad+\int_{[0,t]\times \{\zeta\neq0,|\eta|\leq 1\}} (\phi_s(\zeta(e))-\phi_s(0))  \eta(e) \tilde{N}(ds,de)\\
    &=\int_{[0,t]\times \{\zeta\neq0,|\eta|\leq 1\}} \phi_s(0) \eta(e) \tilde{N}(ds,de)\\&\quad+\int_{[0,t]\times \{\zeta\neq0,|\eta|\leq 1\}} (\phi_s(\zeta(e))-\phi_s(0))  \eta(e) N(ds,de).
\end{align}
Here, the integrals of $\phi_s(\zeta(e))-\phi_s(0)$ are well-defined in this simple case because the set $\{s\;:\;\tilde{\phi}_s\not=\phi_s(0)\}$ is finite. For the same reason, the use of $N$ rather than $\tilde{N}$ in the last term is justified since the compensator $ds\otimes\nu(de)$ occuring in $\tilde{N}(ds,de)$ almost surely does not act on this a.s.\ finite set. Moreover, note that the integral of $\phi_s(0) \eta(e)$ w.r.t.\ $\tilde{N}(ds,de)$ is well-defined by classical stochastic integration theory, due to our integrability assumptions on $\phi(0)$ and $\eta$.

The following result shows that we can continuously extend from such simple strategies to a much larger class of integrands. This class will be characterized in terms of the average small jump given a signal 
\begin{align}\label{defetahat}
    \hat{\eta}(z):= \frac{1}{K(z,\{|\eta|\leq 1\})}\int_{\{\zeta=z,|\eta|\leq 1\}} \eta(e) K(z,de), \quad z \in \zeta(E)\setminus\{0\},
\end{align}
(with the convention that $\hat{\eta}(z):=0$ if $K(z,\{|\eta|\leq 1\})=0$)
and in terms of the remaining uncertainty as measured by the variance
\begin{align}\label{defveta}
    v_{\eta}(z):=\int_{\{\zeta=z,|\eta|\leq 1\}} (\eta(e)-\hat{\eta}(z))^2 K(z,de), \quad z \in \zeta(E)\setminus\{0\}.
\end{align}

\begin{thm}\label{Thm:LambdaIntegration} If $\int_{\zeta(E)\setminus\{0\}}K(z,\{|\eta|\leq 1\})|\hat{\eta}(z)|\mu(dz)<\infty$,
 there is a unique continuous, linear extension of $I=I(\tilde{\phi})$ to the class $\mathcal{I}(Z)$ of $\Lambda$-measurable $\tilde{\phi}=\phi(\Delta Z)$ satisfying almost surely 
 \begin{align}\label{intconditions}
     \int_0^T \left\{\phi_s^2(0)+\int_{\zeta(E)\setminus\{0\}} \left(|\phi_s(z)|K(z,\{|\eta|\leq 1\})|\hat{\eta}(z)| + \phi_s(z)^2v_\eta(z)\right)\mu(dz)\right\} ds<\infty.
 \end{align}
  Here, continuity is meant in in the sense that $(I_t(\tilde{\phi}^n))_{t \in [0,T]} \to 0$ uniformly in probability for any sequence of $\Lambda$-measurable $(\tilde{\phi}^n=\phi^n(\Delta Z))_{n=1,2,\dots}$ for which the corresponding integral in~\eqref{intconditions} vanishes in probability.
  
  The integral process $I(\tilde{\phi})$ is a special semimartingale with Doob-Meyer decomposition $I(\tilde{\phi})=M^{(1)}(\tilde{\phi})+M^{(2)}(\tilde{\phi})+A(\tilde{\phi})$
 into the local martingales
  \begin{align}
      M^{(1)}_t(\tilde{\phi})&:=\int_{[0,t]\times \{\zeta\not=0, |\eta|\leq 1\}} \phi_s(\zeta(e))\hat{\eta}(\zeta(e))\tilde{N}(ds,de), \quad t \in [0,T],\\
      M^{(2)}_t(\tilde{\phi})&:=\int_{[0,t]\times \{\zeta\not=0,|\eta|\leq 1\}} \phi_s(\zeta(e))\left(\eta(e)-\hat{\eta}(\zeta(e))\right)N(ds,de), \quad t \in [0,T],
  \end{align}
  and the absolutely continuous, adapted process
  \begin{align}
      A_t(\tilde{\phi}):=\int_0^t \int_{\zeta(E)\setminus\{0\}} (\phi_s(z)-\phi_s(0))K(z,\{|\eta|\leq 1\})\hat{\eta}(z)\mu(dz) ds, \quad t \in [0,T].
  \end{align}
  If even the expectation of the integral in~\eqref{intconditions} is finite then $M^{(1)}(\tilde{\phi})$ is a martingale with finite expected total variation over $[0,T]$, $M^{(2)}(\tilde{\phi})$ is a square-integrable martingale whose $L^2$-norm at time $T$ is $\E[\int_0^T \int_{\zeta(E)\setminus\{0\}}\phi_s(z)^2v_\eta(z)\mu(dz) ds]^{1/2}$, and $A(\tilde{\phi})$ has finite expected total variation over $[0,T]$.
\end{thm}
\begin{proof}
In the following, let $\tilde{\phi}=\phi(\Delta Z)$ be simple in the sense that $\tilde{\psi}:=\psi(\Delta Z):=\phi(\Delta Z)-\phi(0)$ is bounded and different from zero at most $n$ times almost surely for some $n \in \mathbb{N}$. The integral of such $\tilde{\phi}$ is given by~\eqref{defsimpleint} and can be decomposed as
\begin{align}
    I_t(\tilde{\phi})
    &=\int_{[0,t]\times \{\zeta\neq0,|\eta|\leq 1\}} \phi_s(0) \eta(e) \tilde{N}(ds,de)\\&\quad+\int_{[0,t]\times\{\zeta\neq0,|\eta|\leq 1\}} \psi_s(\zeta(e))\hat{\eta}(\zeta(e))N(ds,de)\label{eq:proofintegration}\\&\quad+\int_{[0,t]\times\{\zeta\neq0,|\eta|\leq 1\}} \psi_s(\zeta(e))(\eta(e)-\hat{\eta}(\zeta(e)))N(ds,de).
\end{align}
Observe that the last term in \eqref{eq:proofintegration}, which corresponds to $M^{(2)}_t(\tilde{\psi})$ in our theorem's notation, is a right-continuous martingale since for any stopping time $\tau\leq T$ its expectation vanishes:
\begin{align}
    &\mathbb{E}\left[M^{(2)}_{\tau}(\tilde{\psi})\right]
    =\mathbb{E}\left[\int_{[0,\tau]}\int_{\{\zeta\neq0,|\eta|\leq 1\}} \psi_s(\zeta(e)) \left(\eta(e)-\hat{\eta}(\zeta(e))\right)\nu(de)ds\right]\\
    &=\mathbb{E}\left[\int_{[0,\tau]}\int_{\zeta(E)\setminus\{0\}} \psi_s(z)\left(\int_{\{\zeta=z,|\eta|\leq 1\}}\eta(e)K(z,de)-K(z,\{|\eta|\leq 1\})\hat{\eta}(z)\right)\mu(dz)ds\right],
\end{align}
which is zero  by definition of $\hat{\eta}(z)$.
Consequently, the finitely many increments of $M^{(2)}(\tilde{\psi})$ are uncorrelated and so 
\begin{align}
    \mathbb{E}\left[\big(M^{(2)}_T(\tilde{\psi})\big)^2\right]=&\mathbb{E}\left[\int_{[0,T]\times\{\zeta\neq0,|\eta|\leq 1\}} \psi^2_s(\zeta(e))(\eta(e)-\hat{\eta}(\zeta(e)))^2N(ds,de)\right]\\
    =&\mathbb{E}\left[\int_{[0,T]}\int_{\{\zeta\neq0,|\eta|\leq 1\}} \psi^2_s(\zeta(e))(\eta(e)-\hat{\eta}(\zeta(e)))^2\nu(de)ds\right]\\
    =&\mathbb{E}\left[\int_{[0,T]}\int_{\zeta(E)\setminus\{0\}} \psi^2_s(z) v_\eta(z)\mu(dz)ds\right].
\end{align}
Using this Itô-isometry, we can extend in the usual way the integral $M^{(2)}_t(\tilde{\psi})$ as an $L^2$-martingale from simple integrands to  any $\Lambda$-measurable $\tilde{\psi}=\psi(\Delta Z)$ for which the preceding expectation is finite. 
By standard localization methods, we can then further expand this construction to $\Lambda$-measurable processes $\tilde{\psi}=\psi(\Delta Z)$ for which the iterated integral in the preceding expectation is merely finite almost surely. This encompasses any $\tilde{\psi}=\tilde{\phi}-\phi(0)$ which can emerge from $\tilde{\phi}$ satisfying condition~\eqref{intconditions}, and it still leads at least to a local martingale.

Focusing on the second term in~\eqref{eq:proofintegration}, we note that 
\begin{align}
    \mathbb{E}&\left[
    \int_{[0,T]\times\{\zeta\neq0,|\eta|\leq 1\}} |\psi_s(\zeta(e))||\hat{\eta}(\zeta(e))|N(ds,de)\right]\\
    &=
    \mathbb{E}\left[\int_{[0,T]\times\{\zeta\neq0,|\eta|\leq 1\}}|\psi_t(\zeta(e))||\hat{\eta}(\zeta(e))|\nu(de)ds\right]\\
    &=\mathbb{E}\left[\int_{[0,T]}\int_{\zeta(E)\setminus\{0\}}|\psi_t(z)|K(z,\{|\eta|\leq 1\})|\hat{\eta}(z)|\mu(dz)ds\right].
\end{align}
Consequently, we can define by standard localization arguments $$\int_{[0,.]\times\{\zeta\neq0,|\eta|\leq 1\}} \psi_s(\zeta(e))\hat{\eta}(\zeta(e))N(ds,de)$$ as an ordinary Lebesgue-integral for any $\tilde{\psi}=\psi(\Delta Z)$ for which the iterated integral in the above expectation is finite almost surely. Again, by our integrability assumption on $\hat{\eta}$, this encompasses any  $\tilde{\psi}=\tilde{\phi}-\phi(0)$ which can emerge from $\Lambda$-measurable $\tilde{\phi}$ satisfying condition~\eqref{intconditions}. Moreover, we can compensate this integral by passing from $N$ to $\tilde{N}$ and obtain its Doob-Meyer decomposition into the bounded variation local martingale $M^{(1)}(\tilde{\psi})$ and the absolutely continuous adapted process $A(\tilde{\psi})$. Also, we find that both parts of this decomposition are of finite expected total variation if the expectation of the integral in~\eqref{intconditions} is finite. Now, with all integrals constructed for both $\tilde{\phi}$ and $\phi(0)$, we are free to consolidate the contributions to $I(\tilde{\phi})$ from integrals w.r.t.\ $\phi(0)$ and arrive at the claimed representation for $I(\tilde{\phi})$.


As a final step in our proof, let us remark that the continuity property of the integral follows by standard arguments which are omitted here for the sake of brevity.
\end{proof}

\subsection{Admissible strategies with jump signals}

Combining the results from the previous two sections, we can now give a concise description of admissible strategies if we assume henceforth that
\begin{align}
&\int_{\zeta(E)\setminus\{0\}}\left(\frac{1}{\overline{\eta}(z)}\vee\frac{1}{|\underline{\eta}(z)|}\right)K(z,\{|\eta|\leq 1\})|\hat{\eta}(z)|\mu(dz)<\infty\label{cond:wealthprocess1}
\end{align}
and
\begin{align}
&\int_{\zeta(E)\setminus\{0\}}\left(\frac{1}{\overline{\eta}(z)}\vee\frac{1}{|\underline{\eta}(z)|}\right)^2v_\eta(z)\mu(dz)<\infty,\label{cond:wealthprocess2}
\end{align}
where $\underline{\eta}(z)$ and $\overline{\eta}(z)$ are the jump bounds from~\eqref{defetaboundssignal}.

\begin{corollary}\label{cor:wealthdynamics}
 Under conditions~\eqref{noarbitragegivensignal}, \eqref{cond:wealthprocess1} and~\eqref{cond:wealthprocess2}, any $\Lambda$-measurable strategy $\tilde{\phi}=\phi(\Delta Z)$ with
 \begin{align}\label{admissibilityconditions}
     \int_0^T \phi_t^2(0)dt<\infty \text{ and } \phi_t(\Delta_t Z) \in \Phi(\Delta_t Z) \text{ for } t\in [0,T] \text{ a.s. }
\end{align}
 admits wealth dynamics $X^{\tilde{\phi}}=X^\phi$ solving~\eqref{wealthdynamics1} which remain nonnegative throughout $[0,T]$ almost surely. In particular, $X^\phi$ is a semimartingale with dynamics $dX^\phi_t=X^\phi_{t-}dR^\phi_t$ where $R^\phi$ is the semimartingale returns process of~\eqref{returnsprocess} whose changes $dR^\phi$ can be decomposed into 
 contributions from large jumps
 \begin{align}
     \int_{\{\eta> 1\}}\phi_t(\zeta(e))\eta(e)N(dt,de),
 \end{align}
 into local martingale contributions from small jumps and the diffusion term
 \begin{align}
     dM^{\phi}_t&:=\int_{\{\zeta=0,|\eta|\leq 1\}}\phi_t(0)\eta(e)\tilde{N}(dt,de)+\int_{\{\zeta\not=0, |\eta|\leq 1\}} \phi_t(\zeta(e))\hat{\eta}(\zeta(e))\tilde{N}(dt,de)\\
      &\qquad+\int_{ \{\zeta\not=0,|\eta|\leq 1\}} \phi_t(\zeta(e))\left(\eta(e)-\hat{\eta}(\zeta(e))\right)N(dt,de)+\phi_t(0)\sigma dW_t
 \end{align}
 and into an absolutely continuous, adapted drift \begin{align}
     dD^{\phi}_t:=\left(r+\phi_t(0)(\kappa-r)+\int_{\zeta(E)\setminus\{0\}} (\phi_t(z)-\phi_t(0))K(z,\{|\eta|\leq 1\})\hat{\eta}(z)\mu(dz)\right)dt.
 \end{align}
\end{corollary}
\begin{proof}
Conditions~\eqref{noarbitragegivensignal}, \eqref{cond:wealthprocess1} and~\eqref{cond:wealthprocess2} ensure that any $\tilde{\phi}=\phi(\Delta Z)$ with~\eqref{admissibilityconditions} yields a process for which all the above contributions can be integrated by classical stochastic integration or via Theorem~\ref{Thm:LambdaIntegration} to form the cumulative returns process $R^\phi$ of~\eqref{returnsprocess} as a semimartingale. For this, let us note that the integrability condition $\int_{\zeta(E)\setminus\{0\}}K(z,\{|\eta|\leq 1\})|\hat{\eta}(z)|\mu(dz)<\infty$ required by this theorem holds here because of~\eqref{cond:wealthprocess1}. Indeed, our standing assumption that $\nu(\{\eta < -1\})=0$ (cf.~\eqref{etalowerbound}) ensures $\underline{\eta}(z)\geq-1$ so that in conjunction with~\eqref{noarbitragegivensignal} we have $1/|\underline{\eta}(z)| \geq 1$. 

Finally, the range restrictions in~\eqref{admissibilityconditions} ensure that $\Delta R^\phi \geq -1$ and we obtain our claim by the well-known properties of stochastic exponentials; cf. Section~II.8 in~\cite{jacodshiryaev}.
\end{proof}
In light of the above result, it is natural to consider 
\begin{align}\label{admissiblestrategies}
\mathcal{A}(Z):=\{\tilde{\phi} \text{ $\Lambda$-measurable satisfying } \eqref{admissibilityconditions}\}     
\end{align} 
as our class of admissible strategies for an investor receiving the jump signal $Z$.

\section{Optimal Investment Strategy}\label{sec:utilitymaximization}

Having defined admissible strategies with jump signals $\mathcal{A}(Z)$ in~\eqref{admissiblestrategies}, let us turn in this section to Merton's optimal investment problem. More precisely, let us examine the optimization problem of maximizing the expected utility from terminal wealth:
\begin{equation}\label{optimizationproblem}
\sup\limits_{\tilde{\phi}\in\mathcal{A}(Z)}\;\mathbb{E}\big[u(X^{\tilde{\phi}}_T)\big]
\end{equation}
for a suitable utility function  $u:\mathbb{R}^+\rightarrow \mathbb{R}$ specifying the investor's risk aversion.

\subsection{Hamilton-Jacobi-Bellman Equation with a jump signal}\label{ChapterHJB}
Our optimization problem (\ref{optimizationproblem}) has the value function 
\begin{equation}\label{valuefct}
v(t,x)=\max\limits_{\phi\in\Phi}\;\mathbb{E}\bigg[u(X^\phi_T)| X^\phi_t=x\bigg], \quad (t,x)\in[0,T]\times\mathbb{R}^+.
\end{equation}
Dynamic programming suggests that this value function solves, in a suitable sense, the so called Hamilton-Jacobi-Bellman (HJB) equation. For the classical case of a predictable information structure, which is included in our framework by setting $Z\equiv0$,  the HJB equation is well known as 
\begin{equation}\label{HJBpredictable}
\begin{aligned}
&\partial_t v(t,x)+rx\partial_x v(t,x)+\sup_{\varphi\in\Phi} \bigg\lbrace \varphi x (\kappa-r) \partial_x v(t,x) +\frac{1}{2}\partial^2_xv(t,x)x^2\varphi^2 \sigma^2\quad\quad\\
&\hspace{5cm}+ \int_{E}\left[v\left(t,x+\varphi x \eta(e)\right)-v(t,x)\right]\nu(de)\bigg\rbrace=0 
\end{aligned}
\end{equation}
on $[0,T]\times\mathbb{R}^+$, with terminal condition $v(T,x)=u(x)$, $x\in\mathbb{R}^+$. 

In order to understand how this is to be transformed when the investor can react to a signal that she receives, let us heuristically investigate when we can expect the martingale optimality principle to work, i.e., under what conditions the value process $V^\phi_t=v(t,X^{\phi}_t)$, $t \in [0,T]$, can be expected to be a super-martingale for arbitrary admissible strategies $\phi \in \mathcal{A}(Z)$ and a martingale for some (then optimal) strategy $\phi^* \in \mathcal{A}(Z)$. 

Recalling the wealth dynamics from Corollary~\ref{cor:wealthdynamics} and assuming $v$ to be smooth, an application of Ito's formula yields 
\begin{align}
V^\phi_t&= v(0,X^\phi_0)+\int_0^t\partial_t v(s,X^\phi_{s-})ds+\int_0^t\partial_x v(s,X^\phi_{s-})dX^\phi_s+\frac{1}{2}\int_0^t\partial^2_xv(s,X^\phi_{s-})d[X^\phi,X^\phi]^c_s\\
&\quad +\sum\limits_{0\leq s\leq t}\left[v(s,X^\phi_{s})-v(t,X^\phi_{s-})-\partial_xv(s,X^\phi_{s-})\Delta_s X^\phi\right]\\
&= v(0,X^\phi_0)+\int_0^t \left[\partial_t v(s,X^\phi_{s-})+\frac{1}{2}\partial^2_xv(s,X^\phi_{s-})\sigma^2(X^\phi_{s-})^2(\phi_s(0))^2\right]ds\\\label{valueprocessdecomposition}
&\quad+\int_0^t\partial_xv(s,X^\phi_{s-})X^\phi_{s-}(dM^{\phi}_s+dD^{\phi}_s)\\
&\quad+\int_{[0,t]\times E} \left[v\left(s,X^\phi_{s-}\left(1+\phi_s(\zeta(e)) \eta(e)\right)\right)-v(s,X^\phi_{s-})\right.\\&\hspace{2.5cm}\left.-\mathbbm{1}_{\{|\eta|\leq1\}}(e)\partial_xv(s,X^\phi_{s-})X^\phi_{s-}\phi_s(\zeta(e)) \eta(e)\right]\Big(\bar{N}(ds,de)+ds\otimes\nu(de)\Big),
\end{align}
where
\begin{align}
    \bar{N}(ds,de):=N(ds,de)-ds\otimes\nu(de)
\end{align}
denotes the fully compensated Poisson measure, including compensation on $\{|\eta|>1\}=\{\eta>1\}$.
For this to be possible, we need to ensure that $v$ satisfies
\begin{align}\label{cond:integrabilityHJB1}
    \int_{\{\eta>1\}}v\left(s,x\left(1+\phi_s(\zeta(e)) \eta(e)\right)\right)\nu(de)<\infty \text{ for any } x>0,
\end{align}
i.e.\ integrable contributions from ``large'' stock price jumps (as will be verified in our case study below).

Let us now collect the drift, i.e. the $ds$-terms in these dynamics. Contributions not related to jumps are 
\begin{align}\label{dsnojumpconnection}
    \partial_t v(s,X^\phi_{s-})+\frac{1}{2}\partial^2_{x}v(s,X^\phi_{s-})\sigma^2(X^\phi_{s-})^2(\phi_s(0))^2+\partial_xv(s,X^\phi_{s-})X^\phi_{s-}(r+(\kappa-r)\phi_s(0));
\end{align}
jump-related $ds$-terms amount to
\begin{align}
    \partial_xv(s,X^\phi_{s-})&X^\phi_{s-}\int_{\zeta(E)\setminus\{0\}}(\phi_s(z)-\phi_s(0))K(z,\{|\eta|\leq 1\})\hat{\eta}(z)\mu(dz)\\
    &+\int_{E} \left[v\left(s,X^\phi_{s-}\left(1+\phi_s(\zeta(e)) \eta(e)\right)\right)-v(s,X^\phi_{s-})\right.\label{dsjumpconnection}\\&\hspace{2.5cm}\left.-\mathbbm{1}_{\{|\eta|\leq1\}}(e)\partial_xv(s,X^\phi_{s-})X^\phi_{s-}\phi_s(\zeta(e)) \eta(e)\right]\nu(de).
\end{align}
Let us focus on the last integral and, more precisely, its contribution from the set $\{\zeta\not=0\}$. Using the disintegration~\eqref{eq:disintegration} of $\nu$, this contribution can be written as
\begin{align}
    \int_{\zeta(E)\setminus\{0\}} &\int_{\{\zeta=z\}} \left[v\left(s,X^\phi_{s-}\left(1+\phi_s(z) \eta(e)\right)\right)-v(s,X^\phi_{s-})\right.\\&\hspace{2.5cm}\left.-\mathbbm{1}_{\{|\eta|\leq1\}}(e)\partial_xv(s,X^\phi_{s-})X^\phi_{s-}\phi_s(z) \eta(e)\right]K(z,de)\mu(dz)\\
    = \int_{\zeta(E)\setminus\{0\}} &\left\{\int_{\{\zeta=z\}} \left[v\left(s,X^\phi_{s-}\left(1+\phi_s(z) \eta(e)\right)\right)-v(s,X^\phi_{s-})\right]K(z,de)\right.\\&\hspace{2.5cm}-\partial_xv(s,X^\phi_{s-})X^\phi_{s-}\phi_s(z) K(z,\{|\eta|\leq1\})\hat{\eta}(z)\Bigg\}\mu(dz),
\end{align}
and we see that the integral over the last line cancels with the integral of $\phi_s(z)$ in~\eqref{dsjumpconnection}.

With $x:=X^\phi_{s-}$, $\varphi^0:=\phi_s(0)$, $\varphi^z :=\phi_s(z)$, the drift-terms thus sum up to
\begin{align}
    \partial_t& v(s,x)+\frac{1}{2}x^2\partial^2_{x}v(s,x)\sigma^2(\varphi^0)^2\\&+x\partial_xv(s,x)\left(r+(\kappa-r)\varphi^0-\int_{\zeta(E)\setminus\{0\}}\varphi^0K(z,\{|\eta|\leq 1\})\hat{\eta}(z)\mu(dz)\right)\\
    &+\int_{\{\zeta=0\}} \left[v\left(s,x\left(1+\varphi^0 \eta(e)\right)\right)-v(s,x)-\mathbbm{1}_{\{|\eta|\leq1\}}(e)x\partial_xv(s,x)\varphi^0 \eta(e)\right]\nu(de)\\
    &+\int_{\zeta(E)\setminus\{0\}} \int_{\{\zeta=z\}}\left[v\left(s,x\left(1+\varphi^z \eta(e)\right)\right)-v(s,x)\right]K(z,de)\mu(dz).
\end{align}

As usual in the martingale optimality principle this has to be non-positive for any admissible strategy $\tilde{\phi}=\phi(\Delta Z) \in \mathcal{A}(Z)$ and to vanish for some (then optimal) strategy. In view of the domain constraints~\eqref{admissibilityconditions}, this leads us to the HJB-equation
\begin{align}
&\partial_t v(t,x)+rx\partial_x v(t,x)\\&+\sup\limits_{\varphi\in\Phi(0)} \bigg\lbrace  x\partial_x v(t,x) \left[(\kappa-r) -\int_{\zeta(E)\setminus\{0\}}K(z,\{|\eta|\leq 1\})\hat{\eta}(z)\mu(dz)\right]\varphi +\frac{1}{2}\sigma^2x^2\partial^2_xv(t,x)\varphi^2 \\
&\qquad\qquad+ \int_{\{\zeta=0\}}\left[v(t,x(1+\varphi \eta(e)))-v(t,x)-\mathbbm{1}_{\{|\eta|\leq 1\}}(e)x\partial_xv(t,x)\varphi \eta(e)\right]\; \nu(de)\bigg\rbrace\label{HJB}\\
&+\int_{\zeta(E)\setminus\{0\}}\sup\limits_{\varphi\in\Phi({z})}\bigg\lbrace \int_{\{\zeta=z\}}\left[v(t,x(1+\varphi  \eta(e)))-v(t,x)\right]\;K(z,de)\bigg\rbrace\mu(dz) \; =0
\end{align}
for $(t,x)\in[0,T]\times\mathbb{R}^+$. Of course, the terminal condition remains unchanged:
\begin{align}
    v(T,x)=u(x), \quad x \in \mathbb R^+.
\end{align}
It is now a standard exercise to formulate an abstract verification theorem that would assume the availability of a `nice' solution $v$ to this HJB-equation and impose some (typically not particularly tangible) regularity conditions to path a way towards a solution to our utility maximization problem~\eqref{optimizationproblem}. Rather than following this well trodden path, we prefer to work out a case study for power utilities in the next section where an explicit solution can be given and a financial-economic discussion of market signals becomes possible once these are fully specified.

\subsection{A solution for power utility}

In order to solve problem (\ref{optimizationproblem}) explicitly, let us restrict to the case where the investor's risk aversion is described by a power utility function: \begin{equation}\label{eq:powerutility}
u(x)=\frac{x^{1-\alpha}}{1-\alpha},\quad x> 0,\quad u(x)=\begin{cases}0,\quad &\alpha\in(0,1)\\
-\infty,\quad &\alpha > 1
\end{cases}
\end{equation}
for some relative risk aversion constant $\alpha \in (0,\infty)\setminus\{1\}$; the case $\alpha=1$ corresponds to $\log$-utility and is omitted for ease of presentation.
\begin{thm}\label{Thm:Verification}
Let the investor have power utility preferences~\eqref{eq:powerutility} and suppose that
\begin{align}
M:=&\sup\limits_{\varphi\in\Phi(0)} \bigg\lbrace \varphi\Big[(\kappa-r) -\int_{\zeta(E)\setminus\{0\}}K(z,\{|\eta|\leq 1\})\hat{\eta}(z)\mu(dz)\Big] -\frac{1}{2}\alpha\varphi^2 \sigma^2 \\&\hspace{1cm}+\int_{\{\zeta=0\}}\big[u(1+\varphi \eta(e))-u(1)-\mathbbm{1}_{\{|\eta|\leq 1\}}(e)\varphi\eta(e) \big]\nu(de)\bigg\rbrace\label{ConstantM_infiniteact}\\
&+ \int_{\zeta(E)\setminus\{0\}}\sup_{\varphi\in\Phi(z)} \bigg\lbrace\int_{\{\zeta=z\}}\big[u(1+\varphi  \eta(e))-u(1)\big]K(z,de)\bigg\rbrace\mu(dz)\quad<\infty.
\end{align}
Then the value function of problem \eqref{optimizationproblem} is given by 
\begin{equation}\label{valuefctpowerutility}
v(t,x)= u(xe^{(r+M)(T-t)})=u(x)\:e^{(1-\alpha)(r+M)(T-t)}, \quad (t,x) \in [0,T] \times (0,\infty).
\end{equation}
and it is indeed a smooth solution solution to the HJB-equation~\eqref{HJB}.  

Moreover, the optimal strategy is of the form $\tilde{\phi}^*_t=\phi^*(\Delta_tZ)$ where  
\begin{align}
\phi^*(0):= \argmax\limits_{\varphi\in\Phi(0)}\:\Big\lbrace&\varphi\Big[(\kappa-r) -\int_{\zeta(E)\setminus\{0\}}K(z,\{|\eta|\leq 1\})\hat{\eta}(z)\mu(dz)\Big] -\frac{1}{2}\alpha\varphi^2 \sigma^2 \\&+ \int_{\{\zeta=0\}}\big[u(1+\varphi \eta(e))-u(1) -\mathbbm{1}_{\{|\eta|\leq 1\}}\varphi\eta(e)\big]\nu(de)\Big\rbrace
\end{align}
specifies the default investment without a signal and where \begin{equation}
\phi^{*}(z)=\argmax\limits_{\varphi\in\Phi(z)} \int_{\{\zeta=z\}}\big[u(1+\varphi  \eta(e))-u(1)\big]K(z,de),\quad z\in z(E)\setminus\{0\}
\end{equation}
specifies what to do optimally when receiving the signal $z\neq 0$.
\end{thm}
\begin{remark}
The condition $M<\infty$ always holds for $\alpha>1$ since the utility function is bounded from above in this case. General conditions for $\alpha \in (0,1)$ are difficult to give as $M$ can become infinite when jumps are ``too favorable'' or signals ``too good''.
\end{remark}
\begin{proof}
Let the function $v(t,x)$, $(t,x)\in[0,T]\times\mathbb{R}^+$ be as given by \eqref{valuefctpowerutility}. It is readily checked that  $v$ is a $C^{1,2}([0,T]\times (0,\infty))$-solution of the HJB-equation~\eqref{HJB} and meets the terminal condition $v(T,x)=u(x)$ for all $x\in\mathbb{R}^+$.  Moreover, the function $v$ satisfies $\lim\limits_{\epsilon\to 0} v(t,\epsilon)=u(0+)$ for every $t\in[0,T]$ and  condition~\eqref{cond:integrabilityHJB1} guaranteeing well-definedness of the value process dynamics in \eqref{valueprocessdecomposition} is implied by \eqref{ConstantM_infiniteact}.

We consider an admissible strategy $\phi$ and let us argue that 
\begin{align}\label{eq:optimalityproof}
\E [u(X^\phi_T)|X^\phi_t=x] \leq v(t,x) \quad\text{for all $(t,x)\in[0,T]\times\mathbb{R}^+.$}
\end{align} 
To this end,  let us consider, for fixed $\epsilon>0$, the process $V^{\phi,\epsilon}_t:=v(t,X^\phi_t+\epsilon)$, $t\in [0,T]$.
By the argument in Chapter~\ref{ChapterHJB}, we know that the drift of the process $V^{\phi,\epsilon}$ is non-positive for 
every admissible $\phi$.
Moreover, since we know that for every $\epsilon>0$ and every admissible $\phi$, the process $V^{\phi,\epsilon}\geq  - C_\epsilon$ is bounded from below for some constant $C_\epsilon>0$, we conclude that $V^{\phi,\epsilon}$ must be a supermartingale as its local martingale part is bounded from below (cf. \eqref{valueprocessdecomposition}). This implies that for all $t\in[0,T]$ and all $\epsilon>0$, we have $$\mathbb{E}[u(X^\phi_T)|X^\phi_t=x]\leq \mathbb{E}[V^{\phi,\epsilon}_T|X^\phi_t=x]\leq \E[V^{\phi,\epsilon}_t|X^\phi_t=x]=v(t,x+\epsilon).$$
We can then send $\epsilon \downarrow 0$  and conclude \eqref{eq:optimalityproof}.

Let us next construct the candidate optimal policy $\phi^*$ by arguing that both arg max used for its description are singletons.
For the maximization in $\phi^*(z)$, $z\neq 0$, it is readily checked that we have upper semi-continuity and strict concavity of the target function and compactness of the domain---and so we get a unique $\argmax$. The same argument applies to $\phi^*(0)$ in case the set $\Phi(0)$ is compact. If $\nu(\{\zeta=0\})>0$, but $\Phi(0)$ is not compact, we can consider the two cases where $\Phi(0)=[a,+\infty)$ and where $\Phi(0)=(-\infty,b]$, $a,b\in\mathbb{R}$, corresponding to the case where there are only positive or only negative jumps. However, since the term $-\frac{1}{2}\alpha\sigma^2\varphi^2$ dominates the jump term in the target function for large $|\varphi|$, the maximum must be attained at some unique point $\phi^*(0)\in\mathbb{R}$ as well. Finally, if $\nu(\{\zeta=0\})=0$, the unique maximum $\phi^*(0)=\argmax\limits_{\varphi\in\mathbb{R}}\{\varphi[\kappa-r-\int_{\zeta(E)\setminus\{0\}}K(z,\{|\eta|\leq 1\})\hat{\eta}(z)\mu(dz)]-\frac{1}{2}\alpha \varphi^2\sigma^2\}$ solves a simple quadratic optimization problem of $\mathbb{R}$.

To complete the proof, we need to show that the value process $V^{\phi^*}$ is a martingale for the optimal strategy.  The argument in Chapter~\ref{ChapterHJB} already shows that the process $V^{\phi^*}$ has zero drift and collecting the local martingale parts from equation~\eqref{valueprocessdecomposition}
gives 
\begin{align}
    &\int_0^t\partial_xv(s,X^\phi_{s-})X^\phi_{s-}dM^{\phi}_s+\int_{[0,t]\times E} \left[v\left(s,X^\phi_{s-}\left(1+\phi_s(\zeta(e)) \eta(e)\right)\right)-v(s,X^\phi_{s-})\right.\\&\hspace{4.5cm}\left.-\mathbbm{1}_{\{|\eta|\leq1\}}(e)\partial_xv(s,X^\phi_{s-})X^\phi_{s-}\phi_s(\zeta(e)) \eta(e)\right]\bar{N}(ds,de).
\end{align}
We observe that the pure jump part of $dM^{\phi}_t$ cancels out with the $\bar{N}$-integral of the $\mathbbm{1}_{\{|\eta|\leq1\}}(e)\dots$-term  in the second line above. This yields the local martingale dynamics
\begin{align}\label{dynamics:locmartingale}
dV^{\phi^*}_t
=V^{\phi^*}_{t-}\bigg(&\phi^*(0)\sigma dW_t+\int_{E} \left[\left(1+\phi^*(\zeta(e)) \eta(e)\right)^{1-\alpha}-1\right]\bar{N}(dt,de)\bigg),
\end{align}
where we used the homotheticity of our candidate value function $v$ in its second argument. 
Both remaining terms are independent and we know for constant fraction $\phi^*(0)$, the corresponding exponential w.r.t.\ the Brownian motion is a martingale. 

Concerning, the integral w.r.t.\ the compensated Poisson random measure, we can apply \cite[Lemma 33.6, p.221]{Sato:99} on exponential L{\'e}vy processes to conclude the martingale property of the corresponding exponential, once we verify the  condition
\begin{align}\label{Cond:Sato}
    \int_E ((1+\phi^*(\zeta(e))\eta(e))^{\frac{1-\alpha}{2}}-1)^2\nu(de)<\infty.
\end{align}
This estimate is going to be established in the remainder of this proof by distinguishing a couple of cases.

For the contribution of the large jumps, we estimate
\begin{align}
    \int_{\{\eta>1\}}& ((1+\phi^*(\zeta(e))\eta(e))^{\frac{1-\alpha}{2}}-1)^2\nu(de)\\\leq &\int_{\{\eta>1\}} [(1+\phi^*(\zeta(e))\eta(e))^{1-\alpha}+1]\nu(de)<\infty
\end{align}
which is finite due to \eqref{ConstantM_infiniteact} and~\eqref{etasquareintegrable}.

For the integral over the set $\{|\eta|\leq 1\}$, let us distinguish between jumps yielding no signal and jumps yielding non-zero signal. By Taylor's approximation and square integrability of small jumps~\eqref{etasquareintegrable}, we can conclude as above
\begin{align}
    \int_{\{|\eta|\leq 1\}\cap \{\zeta=0\}}& ((1+\phi^*(0)\eta(e))^{\frac{1-\alpha}{2}}-1)^2\nu(de)\\
    \leq &\int_{\{|\eta|\leq 1\}\cap \{\zeta=0\}} \mathcal{O}(\eta(e)^2)\nu(de)<\infty.
\end{align}
Finally, for small jumps yielding a non-zero signal, we need to distinguish the cases $\{|\phi^*(\zeta)\eta|\leq 1\}$ and $\{\phi^*(\zeta)\eta> 1\}$. For the first one, we can again apply Taylor's approximation and use the square integrability of small jumps to conclude
\begin{align}
    \int_{\{|\eta|\leq 1\}\cap \{|\phi^*(\zeta)\eta|\leq 1\}\cap \{\zeta\neq0\}} ((1+\phi^*(\zeta(e))\eta(e))^{\frac{1-\alpha}{2}}-1)^2\nu(de)<\infty.
\end{align}
For the integration over the set $\{\phi^*(\zeta)\eta> 1\}$, condition \eqref{ConstantM_infiniteact} provides a finite upper bound in the case $\alpha\in(0,1)$ since here
\begin{align}
    &\int_{\{|\eta|\leq 1\}\cap \{\phi^*(\zeta)\eta> 1\}\cap \{\zeta\neq0\}} ((1+\phi^*(\zeta(e))\eta(e))^{\frac{1-\alpha}{2}}-1)^2\nu(de)\\
    &\quad \leq \int_{\{|\eta|\leq 1\}\cap \{\phi^*(\zeta)\eta> 1\}\cap \{\zeta\neq0\}} ((1+\phi^*(\zeta(e))\eta(e))^{1-\alpha}-1)<\infty.
\end{align}
 Whereas for $\alpha>1$, the term $((1+\phi^*(\zeta(e)))\eta(e))^{\frac{1-\alpha}{2}}-1)^2$ is bounded from above by $1$ on  $\{\phi^*(\zeta)\eta>1\}$, and so
 \begin{align}\label{estimate:proofverification}
     \int_{\{|\eta|\leq 1\}\cap \{\phi^*(\zeta)\eta> 1\}\cap \{\zeta\neq0\}}& ((1+\phi^*(\zeta(e))\eta(e))^{\frac{1-\alpha}{2}}-1)^2\nu(de) \leq \nu(A)
\end{align}
where $A:=\{|\eta|\leq 1\}\cap \{\phi^*(\zeta)\eta> 1\}\cap \{\zeta\neq0\}$. To see that $\nu(A)<\infty$, let us note that
\begin{align}
    \nu(A\cap \{|\phi^*(\zeta)(\eta-\hat{\eta}(\zeta))|\leq 1/2\})
    &\leq \nu(\{|\eta|\leq 1\}\cap \{\zeta\neq0\}\cap\{|\phi^*(\zeta)\hat{\eta}(\zeta)|\geq 1/2\}) \\&\leq \int_{\{|\eta|\leq 1, \zeta\not=0\}} 2|\phi^*(\zeta)\hat{\eta}(\zeta)|\nu(de)\\&= 2\int_{\zeta(E)\setminus\{0\}}|\phi^*(z)|K(z,\{|\eta|\leq1\})|\hat{\eta}(z)|\mu(dz)<\infty
\end{align}
is finite because of~\eqref{cond:wealthprocess1}. Similarly, we find for the complementary set
\begin{align}
\nu(A\cap \{|\phi^*(\zeta)(\eta-\hat{\eta}(\zeta))|> 1/2\})&\leq \nu(\{|\phi^*(\zeta)(\eta-\hat{\eta}(\zeta))|> 1/2\}) \\&\leq \int_{\{\zeta\not=0\}}(2\phi^*(\zeta(e))(\eta(e)-\hat{\eta}(\zeta(e)))^2\nu(de)\\
&\leq 4\int_{\{\zeta(E)\setminus\{0\}}\phi^*(z)^2v_{\eta}(z)\mu(dz)<\infty
\end{align}
because of~\eqref{cond:wealthprocess2}. 

With condition~\eqref{Cond:Sato} fully established, this complete our proof.


\end{proof}

\subsection{Gaussian jumps and signals}
Let us assume the risky asset follows log-normal price dynamics as specified by
\begin{align}\label{assetdynamicscasestudy}
S^1_t=s^1\:\exp\Big(\sigma W_t + (\kappa-\frac{1}{2}\sigma^2)t+\sum\limits_{n=1}^{N_t} (\hat{\sigma}Y_n + \hat{\kappa}-\frac{1}{2}\hat{\sigma}^2)\Big),\quad t\in[0,T];
\end{align}
here, $(N_t)_{t\in[0,T]}$ is a standard Poisson process with intensity $\lambda>0$, $(Y_n)_{n\in\mathbb{N}}$ is an independent sequence of i.i.d.\ standard normal random variables and $\hat{\kappa}\in\mathbb{R},\hat{\sigma}>0$. Note that our use of $\kappa$ here deviates from our model specification~\eqref{assetdynamics}. This is for notational convenience and due to the compound Poisson process setting here where we do not have to be concerned about compensators and stochastic integration. All computations henceforth refer to $\kappa$, $\hat{\kappa}$, etc.\ as specified in~\eqref{assetdynamicscasestudy}.

To model our signal, we consider furthermore an independent sequence of i.i.d.\ Bernoulli random variables $(\delta_n)_{n \in \mathbb N}$ with $p=\mathbb P[\delta_n=1]=1-\mathbb P[\delta_n=0] \in [0,1]$ and another independent i.i.d.\ sequence of standard normal random variables  $(\epsilon_n)_{n\in\mathbb{N}}$, and we put
$$Z_t=\sum_{n=1}^{N_t}\delta_n(\rho Y_n+\sqrt{1-\rho^2}\epsilon_n), \quad t \in [0,T].$$
With probability $p$, this signal process will jump together with the compound Poisson process $\sum_{n=1}^{N_.} (\hat{\sigma}Y_n + \hat{\kappa}-\frac{1}{2}\hat{\sigma}^2)$ driving the price process $S^1$ and its jumps have correlation $\rho$ with the price shocks.

In the marked Poisson point process notation from before, this corresponds to choosing the Polish space $E=\mathbb{R}^2\times \{0,1\}$ equipped with the $\sigma$-field $\mathcal{E}=\mathcal{B}(\mathbb{R}^2)\otimes 2^{\{0,1\}}$ and letting $N(dt,du)$ be a Poisson random measure with compensator $\nu=\lambda N(0,1)\otimes N(0,1) \otimes \text{Ber}(p)$; the mappings $\eta,\zeta:E\to\mathbb R$ are given by $\eta(e)=\exp(\hat{\sigma}e_1+\hat{\kappa}-\frac{1}{2}\hat{\sigma}^2)-1$ and $\zeta(e)=e_3(\rho e_1+\sqrt{1-\rho^2}e_2)$ for $e=(e_1,e_2,e_3)\in E$. 

When $|\rho|=1$, the investor knows exactly what stock price jump is about to happen and thus has obvious arbitrage opportunities that render the utility maximization trivial; note also that this violates the no arbitrage condition~\eqref{noarbitragegivensignal}.


For $|\rho|<1$, we find that $\nu$ disintegrates into $\mu(dz)=\lambda p N_{0,1}(dz)$ and $K(z,de)=N_{\rho z,1-\rho^2}(de_1) \otimes \mathrm{Dirac}_{z}(de_2) \otimes \mathrm{Dirac}_1(de_3)$. So, the investor cannot with certainty rule out any open jump range  and for any signal $z\in \mathbb R \setminus \{0\}$ we find
\begin{align}
\Phi(z)=[0,1].
\end{align}
Interestingly, the form of $\Phi(0)$ depends on the signal probability $p$. If $p=1$, the investor can be sure to be warned about any stock price jump and thus can rewind any leveraged position upon receiving a signal and in the meantime choose her investment fraction from all of $\Phi(0)=\mathbb R$. For $p<1$ though, arbitrary jumps can happen at any time and so the no-bankruptcy requirement $X^\phi \geq 0$ reduces the admissible investment choices to $\Phi(0)=[0,1]$. It is now easy to check that $M<\infty$ for the constant $M$ of~\eqref{ConstantM_infiniteact}. 

The optimal portfolio $\phi^*(0)$ at times without signal (i.e. when $\Delta_t Z=0$) is thus given by
\begin{align}\label{phi0example}
\phi^*(0)
&=\argmax\limits_{\varphi\in \Phi(0)} \Big\lbrace(\kappa-r)\varphi-\frac{1}{2}\alpha\sigma^2\varphi^2 \\&\qquad +\lambda\frac{1-p}{1-\alpha}\int_{\mathbb{R}}[(1+\varphi (\exp(\hat{\sigma}x+\hat{\kappa}-\frac{1}{2}\hat{\sigma}^2)-1))^{1-\alpha}-1]N_{0,1}(dx)\Big\rbrace.
\end{align}
Similarly, upon receiving the signal $z \in \mathbb{R}$, the investor will adjust her investment fraction to
\begin{align}
\phi^*(z)&=
\argmax\limits_{\psi\in[0,1]}\frac{1}{1-\alpha}\int_{\mathbb{R}}[(1+\psi (\exp(\hat{\sigma}x+\hat{\kappa}-\frac{1}{2}\hat{\sigma}^2)-1))^{1-\alpha}-1]N_{\rho z,1-\rho^2}(dx).
\end{align} 
Lacking an analytic description of the above maximizers, let us resort to numerical methods to study them in financial-economic terms in the next section.

\subsection{Numerical illustration and financial-economic discussion}
In the following, let us discuss the above example with the parameter choices 
\begin{align}
\alpha=2,\; r=0,\; \kappa=0.08,\; \sigma=0.3,\; \hat{\kappa}=0,\; \hat{\sigma}=0.1\; \text{ and } \lambda=4.    
\end{align} 
Since the optimal strategy cannot be computed explicitly, we will apply numerical integration and optimization methods in our illustrations. 
\begin{figure}[t]
\begin{center}
\includegraphics[scale=0.6]{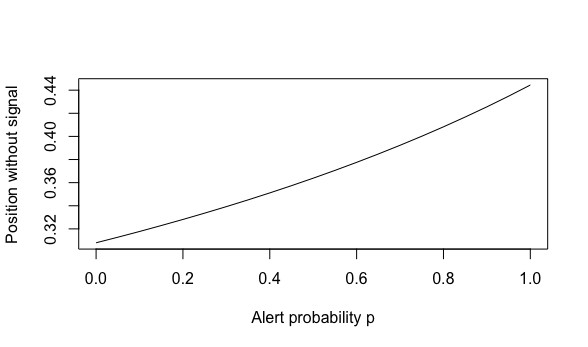}\end{center}
\vspace{-0.2cm}
\caption{The optimal investment fraction~$\phi^*(0)$ at times without signal as it changes with the signal probability~$p$ for risk aversion $\alpha=2$.}\label{fig1}
\end{figure}

Figure~\ref{fig1} describes how the optimal investment fraction with no incoming signal, $\phi^*(0)$, depends on the probability $p$ of receiving a jump alert.  For $p=0$, the investor finds herself in the classical Merton investment problem without jump signals and thus opts for the classical Merton ratio $\phi^{\text{Merton}} \approx 0.31$. Moreover,  we observe that the higher the probability $p$ of being warned of impending price jumps, the more aggressive the investor will be. The reason for this is that, with more alerts on jumps, the investor has more opportunities to hedge against price jumps. 
\begin{figure}[t]
\begin{center}
\includegraphics[scale=0.33]{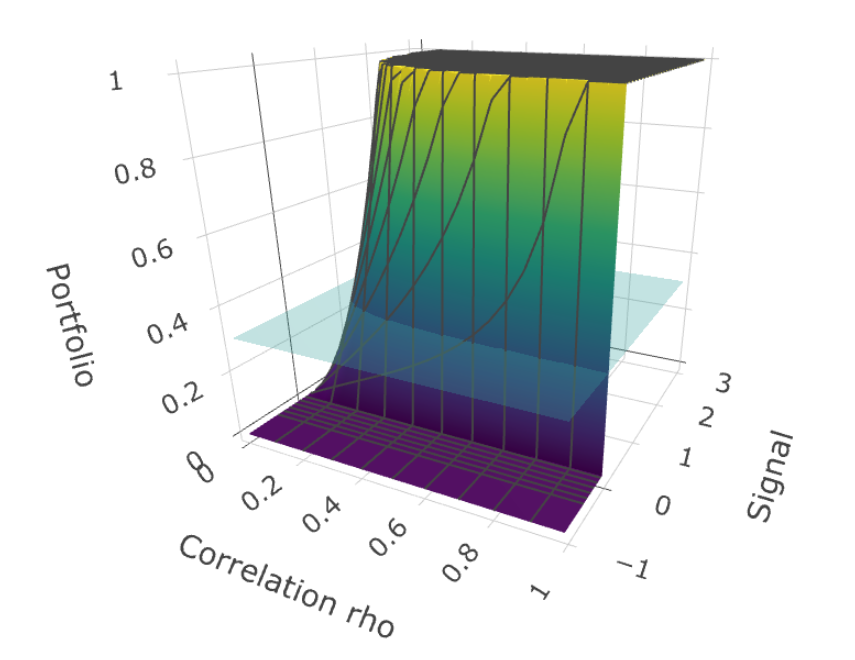}\end{center}
\vspace{-0.2cm}
\caption{The Merton portfolio $\phi^*$ for $p=0$ (transparent plane) and the signal-dependent optimal strategy $\phi^*(z)$ are plotted in dependence of the correlation $\rho$ and the arriving signal $z$ for risk aversion $\alpha=2$.}\label{fig2}
\end{figure}

Figure \ref{fig2} shows how the optimal investment ratio $\phi^*(z)$ depends on the received signal $z$ and on the correlation $\rho$ between signal and price shock; the Merton ratio $\phi^*$ for $p=0$, i.e. without signals, is included as a transparent plane.  As one would expect, the investment ratio $\phi^*(z)$ is non-decreasing in the signal $z$ since we have a positive correlation coefficient $\rho$. Moreover, for  sufficiently positive jump signals, the investor will put all her money into the risky stock. However, for any negative signal, the investor will completely withdraw her funds from the risky stock. Interestingly, a disinvestment can happen also for slightly positive signals, especially for low correlation $\rho$, which is an effect of risk aversion that makes the investor wary of a shock even if it promises to be positive on average. Another observation is that the more reliable the signal, i.e.\ the higher $\rho$, the steeper the investment curve as a function of the signal $z$. Finally, also for a per se completely meaningless signal which has no correlation with the stock price jump ($\rho=0$), the investor still takes action upon receiving a jump alert: she completely disinvests to $\phi^*(z)=0$. The reason for this is that, having learned about the impending shock and aware of its mean $\hat{\kappa}=0$, our risk-averse investor insulates herself for the next moment from any exposure to the stock price shock.
\begin{figure}[t]
\begin{center}
\includegraphics[scale=0.3]{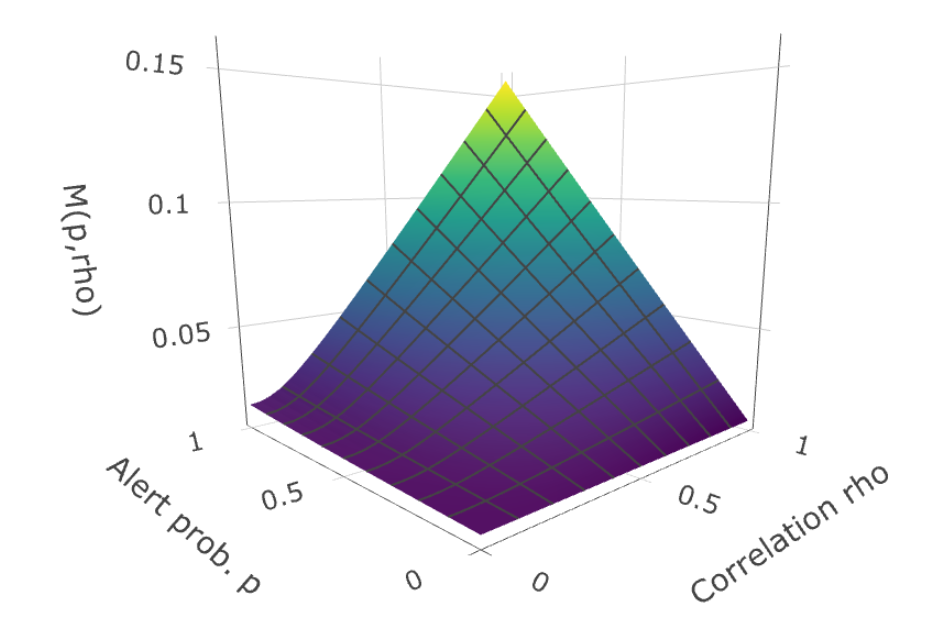}\end{center}
\vspace{-0.2cm}
\caption{The constant $M$ is plotted in dependence of the alert probability $p$ and the correlation $\rho$ for risk aversion $\alpha=2$. }\label{fig3}
\end{figure}

Figure~\ref{fig3} shows the constant $M$ determining the certainty equivalent growth rate~\eqref{ConstantM_infiniteact} as it depends on signal quantity (as represented by the alert probability $p$) and signal quality (as represented by the correlation $\rho$).  Without the chance for a signal, i.e.\ for $p=0$, it of course coincides the growth rate in the usual Merton jump diffusion problem. However, already for a meaningless signal with correlation $\rho=0$, we detect the effect of the hedging opportunity (particularly for large $p$) which is due to the hedging opportunities afforded by jump alerts. More generally, one can conclude that, for the present parameter choice, improving the quantity or the quality of the signal both have pretty similar effects on the expected utility of the investor's strategy: Both affect the growth rate $M$ in an approximate linear way with similar slopes. Hence, when forced to decide whether to improve signal quantity or signal quality, our investor should pursue the cheaper option. 

\begin{figure}[t!]
\includegraphics[scale=0.5]{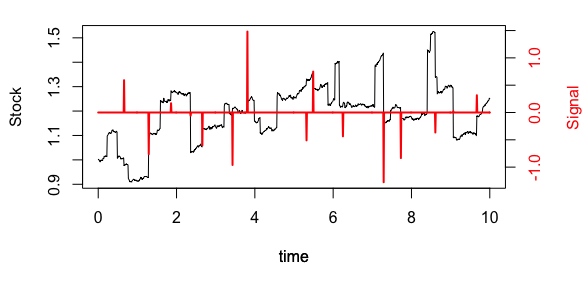}\\ 
\includegraphics[scale=0.5]{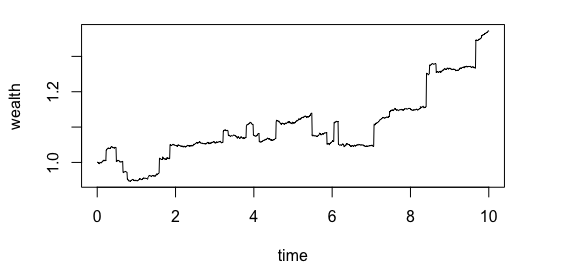}\\ 
\includegraphics[scale=0.5]{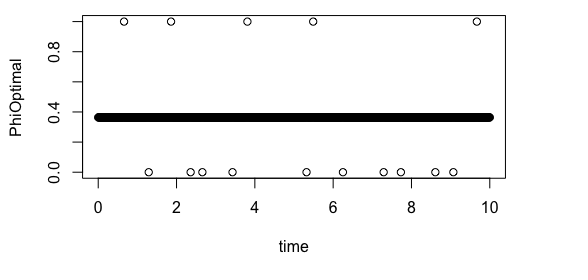}\caption{ Trajectories for stock prices, signals and wealth process arising from optimal investment fractions depicted in the lower panel for $p=0.5$ and $\rho=0.8$ for risk aversion $\alpha=2$.}\label{fig4}
\end{figure}
\begin{figure}[t!]
\includegraphics[scale=0.5]{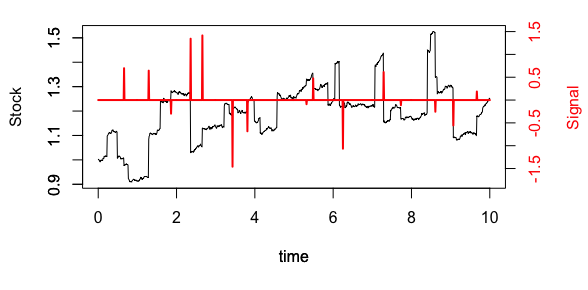}\\ 
\vspace{-0.2cm}
\includegraphics[scale=0.5]{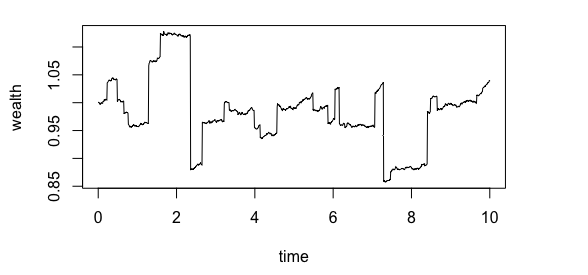}\\ 
\includegraphics[scale=0.5]{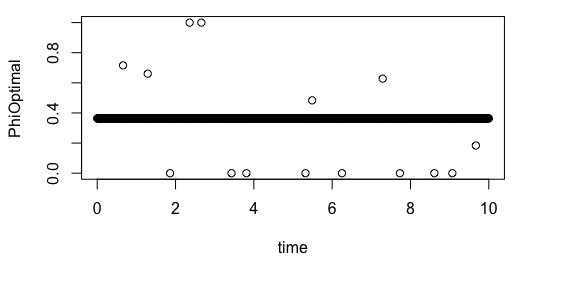}\caption{ Same investment scenario as in figure \ref{fig4} for an investor receiving a less reliable signal with $\rho=0.2$ for risk aversion $\alpha=2$.}\label{fig5}
\end{figure}  

Finally, Figures~\ref{fig4} and~\ref{fig5} illustrate typical paths for the stock price evolution $S^1$, the signals $\Delta Z$ processes and the investor's wealth dynamics $X^{\phi^*}$ along with the corresponding investment ratios $\phi^*(\Delta Z)$, in a high ($\rho=0.8$) and a low ($\rho=0.2$) correlation setting, respectively; the probability of being alerted is the same in both cases ($p=0.5$). For a reliable signal with high correlation to the jump shocks, we find that investment ratios given a signal are typically either $0$ or $1$, reflecting the investor's strong trust in the signal's directional information. For a less reliable signal with low correlation, a more nuanced approach is taken and even upon receiving a positive signal the investor will withdraw some money from the stock if this is not sufficiently high (e.g. at $t\approx 9.5$). 

Of course, even a rather reliable signal ($\rho=0.8$) may produce misleading information, for instance for the upward jump at time $t\approx 1.5$ of which our investor cannot profit as she disinvested in this moment. More often, though, the high quality signal will allow her to make a very high profit, e.g.\ at time $t\approx 9.5$, or to avoid a loss, e.g.\ at $t\approx 7.2$ or $t\approx 7.8$. By contrast, the paths in Figure~\ref{fig5} for a low correlation signal are characterized by two major losses, at $t\approx 2.2$ and $t\approx 7.2$, due to an erroneous ``buy'' signal. On the upside of course, one can observe that the investor does not invest all of her money into the risky stock at the treacherous time $t \approx 7.2$ since she knows that she cannot fully trust her signal. Moreover, she can make a small profit at time $t\approx 9.5$ as she does not withdraw all of her money from the stock, despite having received only a slightly positive signal. 

Finally, let us comment briefly on what changes when our investor becomes more risk averse. All other parameters the same, Figures \ref{fig2_alpha10} and \ref{fig3_alpha10} illustrate for risk aversion $\alpha=10$ the optimal portfolio choices upon receiving a signal as well as the resulting certainty equivalent growth rate $M$ for varying signal quality $\rho$ and quantity $p$.
\begin{figure}[t]
\begin{center}
\includegraphics[scale=0.33]{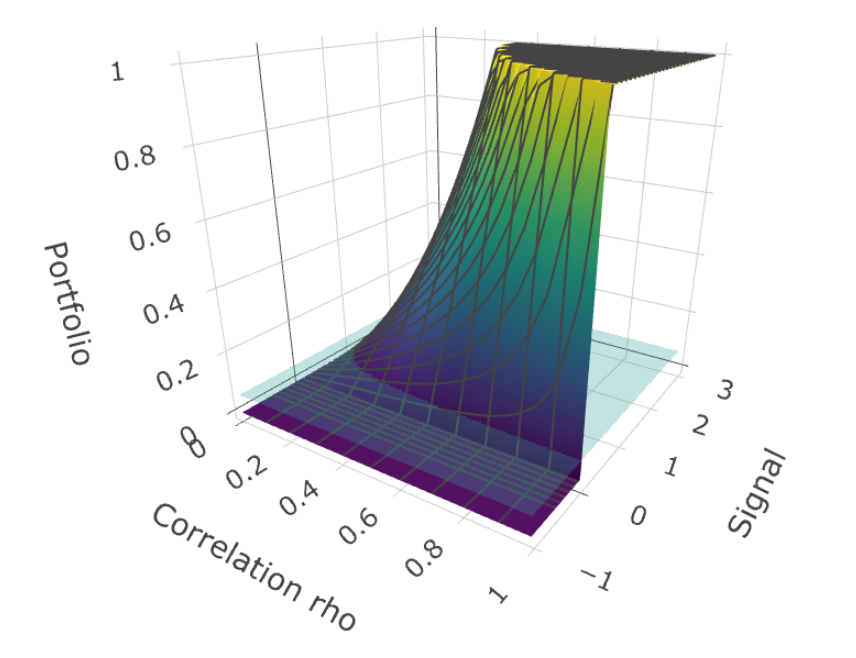}\end{center}
\caption{The Merton portfolio $\phi^*(0)$ for $p=0$ (transparent plane) and the signal-dependent optimal strategy $\phi^*(z)$ are plotted in dependence of the correlation $\rho$ and the arriving signal $z$ in the high risk aversion case $\alpha=10$.}\label{fig2_alpha10}
\end{figure}

 In comparison to the case of $\alpha=2$ in Figure \ref{fig2}, the investment ratio $\phi^*(z)$ is increasing considerably less quickly with the signal as well as with the correlation coefficient $\rho$, a reflection of the investor's aversion for bold bets even when given a signal. As a consequence, the investor still withholds some of her money from the stock even for rather positive signals, and mildly positive signals (let alone negative) ones make her disinvest from the stock more frequently.

\begin{figure}[t]
\begin{center}
\includegraphics[scale=0.3]{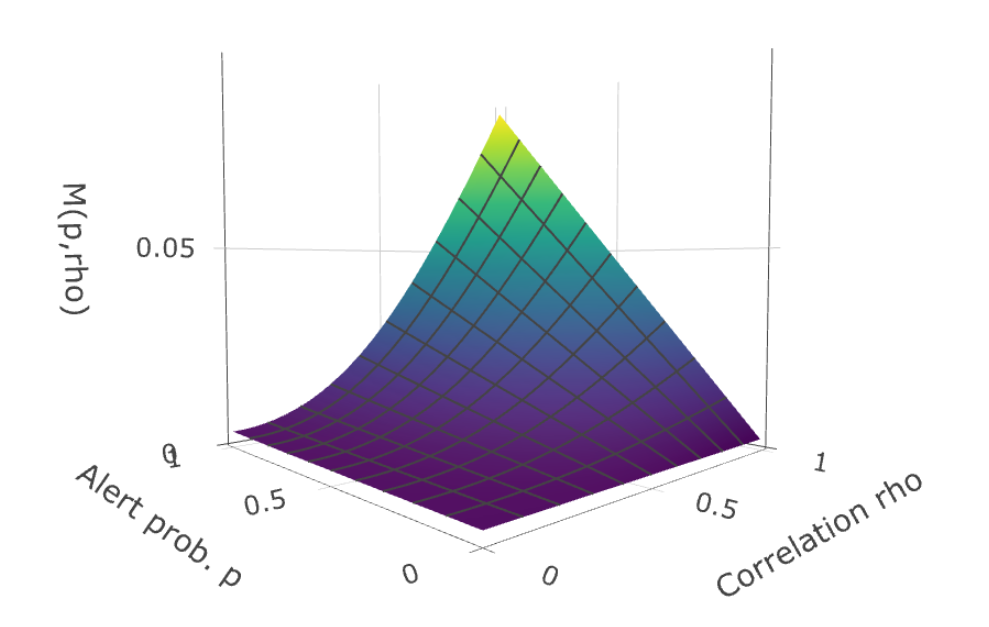}\end{center}
\caption{The constant $M$ is plotted in dependence of the alert probability $p$ and the correlation $\rho$ in the case of high risk aversion $\alpha=10$.}\label{fig3_alpha10}
\end{figure}
For $\alpha=2$ we already noted that improving quantity $p$ or quality $\rho$ of the signal have similar effects on the investor's expected utility. As illustrated by Figure~\ref{fig3_alpha10}, this changes for high risk aversion $\alpha=10$: Here, the increase from changes in signal quantity $p$ remains constant, whereas the effect of changes in $\rho$ exhibits convexity. In other words, we find that the expected marginal utility of improving the signal's quality is essentially constant in the alert probability $p$, but increases with improvements in the signal quality $\rho$. This effect is the more pronounced the higher the alert probability. Thus, a very risk averse investor will assign ever higher value to further improvements of an already reliable signal, whereas opportunities to obtain more signals will not be pursued as vigorously.

\bibliography{BibliographyMerton}

\begin{thebibliography}{10}

\bibitem{AmendingerImkellerSchweizer:98}
J.~Amendinger, P.~Imkeller, and M.~Schweizer.
\newblock Additional logarithmic utility of an insider.
\newblock {\em Stochastic Process. Appl.}, 75(2):263--286, 1998.

\bibitem{Back:92}
K.~Back.
\newblock Insider trading in continuous time.
\newblock {\em Review of Financial Studies}, 5:387--409, 1992.

\bibitem{EkrenEtAl:21}
K.~Back, F.~Cocquemas, I.~Ekren, and A.~Lioui.
\newblock Optimal transport and risk aversion in kyle's model of informed
  trading, 2021.

\bibitem{bankbesslichlenglarts}
P.~Bank and D.~Besslich.
\newblock On lenglart's theory of meyer-sigma-fields and el karoui's theory of
  optimal stopping.
\newblock arXiv e-prints, 2019.

\bibitem{bankbesslichmodelling}
P.~Bank and D.~Besslich.
\newblock Modelling information flows by meyer-$\sigma$-fields in the singular
  stochastic control problem of irreversible investment.
\newblock {\em The Annals of Applied Probability}, 2020.

\bibitem{ElKaroui:81}
N.~El~Karoui.
\newblock Les aspects probabilistes du controle stochastique.
\newblock {\em Hennequin P.L. (eds) Ecole d’Eté de Probabilités de
  Saint-Flour IX-1979. Lecture Notes in Mathematics}, 876, 1981.
\newblock Springer, Berlin, Heidelberg .

\bibitem{Sato:99}
K.~iti Sato.
\newblock {\em Lévy Processes and Infinitely Divisible Distributions}.
\newblock Cambridge University Press, 1999.

\bibitem{jacodshiryaev}
J.~Jacod and A.~N. Shiryaev.
\newblock {\em Limit Theorems for Stochastic Processes}.
\newblock Springer, Berlin Heidelberg, 2nd edition, 2003.

\bibitem{Kyle:85}
A.~S. Kyle.
\newblock Continuous auctions and insider trading.
\newblock {\em Econometrica}, 53:1315--1335, 1985.

\bibitem{Lenglart}
E.~Lenglart.
\newblock Tribus de meyer et th\'eorie des processus.
\newblock {\em S\'eminaire de probabilit\'es de Strasbourg}, 14:500--546, 1980.

\bibitem{Merton71}
R.~C. Merton.
\newblock Optimum consumption and portfolio rules in a continuous-time model.
\newblock {\em Journal of Economic Theory}, 3(4):373--413, 1971.

\bibitem{PikovskyKaratzas:96}
I.~Pikovsky and I.~Karatzas.
\newblock Anticipative portfolio optimization.
\newblock {\em Advances in Applied Probability}, 28(4):1095–1122, 1996.

\bibitem{oksendalsulem}
B.~Øksendal and A.~Sulem.
\newblock {\em Applied Stochastic Control of Jump Diffusions}.
\newblock Springer, 01 2007.

\end{thebibliography}
\bibliographystyle{abbrv}
\end{document}